\documentclass[a4paper,american,reqno]{amsart}

\usepackage{babel}
\usepackage[utf8]{inputenc}
\usepackage[T1]{fontenc}
\usepackage{xspace}
\usepackage{accents}
\usepackage{csquotes}
\usepackage{amssymb}
\usepackage{todonotes}
\usepackage{microtype}
\usepackage[style = authoryear,
            maxbibnames = 100,
            maxcitenames = 2,
            giveninits = true,
            isbn = false,
            url = true,
            backend = bibtex]{biblatex}
\usepackage[colorlinks,
            citecolor=blue,
            urlcolor=blue,
            linkcolor=blue]{hyperref} 
\usepackage{soul}
\vfuzz \hfuzz
\makeatletter
\patchcmd{\@settitle}{\uppercasenonmath\@title}{\scshape\large}{}{}
\patchcmd{\@setauthors}{\MakeUppercase}{\scshape\normalsize}{}{}
\makeatother

\newtheorem{theorem}{Theorem}[section]

\newtheorem{lemma}[theorem]{Lemma}
\newtheorem{proposition}[theorem]{Proposition}
\newtheorem{definition}[theorem]{Definition}
\theoremstyle{definition}
\newtheorem{remark}[theorem]{Remark}

\newtheorem{example}[theorem]{Example}

\providecommand{\U}[1]{\protect\rule{.1in}{.1in}}
\providecommand{\dom}{\mathop{\rm dom}\nolimits}

\providecommand{\argmax}{\mathop{\rm argmax}\nolimits}

\providecommand{\tto}{\mathop{\rightrightarrows}\nolimits}

\newcommand{\GNEP}{\text{GNEP}}

\newcommand{\R}{\mathbb{R}}

\newcommand{\N}{\mathbb{N}}
\newcommand{\Q}{\mathbb{Q}}

\newcommand{\define}{\mathrel{{\mathop:}{=}}}

\newcommand{\defset}[3][\defsep]{\set{#2#1#3}}
\newcommand{\Defset}[3][\defsep]{\Set{#2#1#3}}
\newcommand{\set}[1]{\{#1\}}
\newcommand{\Set}[1]{\left\{#1\right\}}

\bibliography{Refs}

\title[{Are SLMF Problems Particular Cases of Bilevel
  Optimization?}]{Are Single-Leader-Multi-Follower Problems\\Particular
  Cases of Bilevel Optimization?}

\author[D. Aussel, D. Salas, M. Schmidt]%
{Didier Aussel, David Salas, Martin Schmidt}

\address[D. Aussel]{Lab. PROMES UPR CNRS 8521, University of Perpignan via Domitia, Tecnodud, Perpignan 66100, France}
\email{aussel@univ-perp.fr}

\address[D. Salas]{Instituto de Ciencias de la Ingeniería,
  Universidad O'Higgins, Rancagua, Chile}
\email{david.salas@uoh.cl}

\address[M. Schmidt]{%
  Trier University,
  Department of Mathematics,
  Universitätsring 15,
  54296 Trier,
  Germany}
\email{martin.schmidt@uni-trier.de}

\date{\today}

\begin{document}

\begin{abstract}
  For more than 20 years, single-leader multi-follower (SLMF) problems
  have been treated as a more difficult class of
  problems in comparison to bilevel optimization. In this work we
  challenge this common belief, studying whether SLMF problems can be
  reformulated as equivalent bilevel optimization problems. We consider both
  optimistic and pessimistic formulations. For the optimistic case the
  answer is positive: We derive an equivalent bilevel problem by
  representing the followers' equilibria as fixed points of the
  solution mapping of a single optimization problem. For pessimistic
  problems, the situation is different. We show that every optimistic
  trilevel problem admits an equivalent pessimistic SLMF formulation
  with two followers. Conversely, for every pessimistic SLMF problem
  verifying a mild attainability assumption we establish an equivalent
  optimistic trilevel reformulation. The reformulations are easy,
  readily tractable, and they preserve convex and linear data. Thus,
  unless $\text{NP}=\Sigma_2^p$, we deduce as byproduct that multiple
  followers induce an unexpected complexity gap between optimistic and
  pessimistic approaches, which is not the case when considering one
  follower alone.
\end{abstract}

\keywords{Bilevel optimization, Single-leader multi-follower games,
  Complexity theory, Optimistic vs.\ pessimistic}
\subjclass[2020]{%
  90Cxx, 
  91Axx, 
  68Q15, 
  90C33
}

\maketitle

\section{Introduction}\label{sec:intro}

Bilevel optimization deals with optimization problems in which some
variables are required to solve another optimization problem. These
problems model hierarchical interactions between two decision makers,
the leader and the follower. The leader makes a decision first and
anticipates the optimal response of the follower, whose problem depends
on the leader's decision. If several followers interact non-cooperatively
at the lower level, one obtains a single-leader multi-follower (SLMF)
problem. Each follower then solves an optimization problem that also
depends on the strategies of the other followers. The resulting
lower-level problem is a Nash equilibrium problem or, if the followers'
feasible sets depend on the other followers' strategies, a generalized
Nash equilibrium problem (GNEP). Hence, SLMF problems extend the usual
single-leader single-follower (SLSF) setting of bilevel optimization.

While simply quoted by \textcite{vstack} in his pioneering work, such combinations of hierarchical optimization and equilibrium problems have been studied in the seminal work of
\textcite{PangFukushima:2005:QVI}, and can be traced back to
\textcite{Sherali-Syster-Murphy:1983}. For an overview of single- and
multi-leader-follower games, including existence results, reformulations,
and solution methods, we refer to \textcite{AusselSvensson2020:Short}.
A tutorial on solving SLMF problems using SOS1 reformulations is given in
\textcite{AusselEtAl:2025:Tutorial}. A central issue of these models is
that the followers' equilibrium need not be unique. In the optimistic
case, the leader selects the most favorable equilibrium among those
satisfying the upper-level constraints. In the pessimistic case, the
leader accounts for the worst equilibrium outcome and requires the
coupling constraints to hold for every equilibrium. This distinction
already arises in bilevel optimization
\parencite{Dempe:2002:Foundations,BeckLjubicSchmidt:2027:Linear}
and is equally relevant for SLMF problems.

SLMF problems appear more difficult than their SLSF counterparts because
the leader must anticipate a system of mutually dependent optimization
problems. In general, minimizing the sum of the followers' objectives
over their jointly feasible strategies does not characterize the
equilibria of the game. Moreover, convexity of each follower's problem
does not ensure convexity of the equilibrium set, and even uniqueness of
each individual best response does not ensure uniqueness of the
equilibrium. These properties make the analysis and solution of SLMF
problems challenging. However, they do not settle whether SLMF problems
are actually harder than bilevel problems. In particular, the question
remains whether an SLMF problem can be reformulated as a bilevel problem
while preserving properties such as convexity or linearity of the data.

In recent years, several papers have shown that problem classes that
are usually studied separately can be related by suitable
reformulations. \textcite{Zeng:2020:APractical} develops a scheme for
pessimistic bilevel optimization based on a tight optimistic bilevel
relaxation and a procedure for recovering a feasible optimal solution.
This allows one to apply methods for optimistic bilevel optimization
to pessimistic problems. For solvable bilevel problems with rational
linear data, \textcite{HenkeEtAl:2025:OnCoupling} show that coupling
constraints can be removed while preserving linearity and the globally
optimal solutions after projection onto the original variables.
\textcite{HenkeEtAl:2026:CouplingPessimistic} prove corresponding results
for pessimistic linear bilevel problems and also derive optimistic
reformulations without coupling constraints. Finally,
\textcite{LefebvreEtAl2025:One-Extra} show that, under suitable
assumptions, GNEPs can be reformulated as Nash equilibrium problems by
adding one player. These reformulations preserve linear or convex
structure for the respective problem classes. The common idea of these
works is to use additional variables or players to obtain a formulation
that allows one to apply theory and methods from another problem class.

In this paper, we study whether SLMF problems with a lower-level GNEP
can be reformulated as bilevel optimization problems. We first consider
the optimistic case. Inspired by the use of copied follower variables in
\textcite{Zeng:2020:APractical}, we represent the followers' equilibria as
fixed points of the solution mapping of a single optimization problem.
We then introduce a copy of the followers' strategy vector at the upper
level and link both copies by equality constraints. This yields an
equivalent optimistic bilevel problem without any assumption on the
structure of the problem.
We show that the reformulation preserves convex and linear data.
For explicitly encoded rational linear data, the
reformulation can be constructed in polynomial time and has polynomial
encoding length.

For pessimistic SLMF problems, we derive an equivalent optimistic
trilevel optimization problem under an attainability assumption on
truncated constraint-violation functions. The truncation allows us to
handle unbounded constraint violations, and we show that the
attainability assumption always holds under linear data. We also prove
the reverse result: every optimistic trilevel problem can be
reformulated as a pessimistic SLMF problem. To this end, we encode the
feasible region of the middle- and lower-level problems through a
two-player GNEP and impose optimality by robust comparisons with its
equilibria. Both reformulations preserve linear data and can be
constructed in polynomial time with polynomial encoding length for
rational linear data. We also discuss the assumptions and modifications
needed to preserve convex data.

Finally, we analyze the complexity of SLMF problems with rational linear
data. Using the results of
\textcite{Jeroslow1985:Polynomial,BUchheim2023:BilevelinNP,Sugishita-Carvalho:2026},
we prove that the decision problem for optimistic SLMF problems with
rational linear data is NP-complete, whereas its pessimistic counterpart
is $\Sigma_2^p$-complete. The latter result continues to hold without
coupling constraints, already for a two-follower game. We prove this by
reducing the DeNegre bilevel knapsack problem, whose complexity was
established by \textcite{Caprara-et-al:2014:Knapsack}. We further give a
linear example showing that direct exact penalization of the robust
coupling constraints can fail. Thus, unless $\mathrm{NP}=\Sigma_2^p$,
there is no polynomial-time reduction of the general pessimistic SLMF
decision problem with rational linear data to the corresponding
optimistic linear bilevel decision problem. The optimistic and
pessimistic SLMF formulations therefore have different relationships
with bilevel optimization even when all data are linear.
This contrast was unexpected because our starting point was the scheme
of \textcite{Zeng:2020:APractical} for solving pessimistic bilevel
problems through optimistic bilevel optimization.

The remainder of the paper is organized as follows.
Section~\ref{sec:SLMF} introduces the followers' equilibrium problem,
the optimistic and pessimistic SLMF formulations, and the notions of
convex and linear data used throughout the paper.
Section~\ref{sec:optimistic} develops the bilevel reformulation of
optimistic SLMF problems.
Section~\ref{sec:pessimistic-to-three-level} establishes the
reformulations between pessimistic SLMF problems and optimistic trilevel
problems.
Section~\ref{sec:complexity-under-linear-data} derives the complexity
results under rational linear data and studies the case without coupling
constraints.
Section~\ref{sec:conclusions} concludes the paper.

\section{Single-Leader-Multi-Follower Problems}\label{sec:SLMF}

For a natural number $m\in\N$, we write $[m]\define\{1,\ldots,m\}$.
We consider one leader deciding the variable~$x\in X\subseteq\R^p$
and a set $I \define [N]$ of followers.
Each follower $i\in I$ decides a variable $y_i\in Y_i\subseteq\R^{q_i}$.
We write
\begin{equation*}
  Y\define\prod_{i\in I}Y_i,
  \quad
  Y_{-i}\define\prod_{j\in I\setminus\{i\}}Y_j,
  \quad
  q\define\sum_{i\in I}q_i.
\end{equation*}
For $y=(y_i)_{i\in I}\in Y \subseteq \R^q$, the vector containing the
decisions of all followers other than follower $i$ is denoted by
$y_{-i}$.

\subsection{The Followers' Equilibrium Problem}
\label{subsec:followers}

For each $x\in X$ and $y_{-i}\in Y_{-i}$, follower $i$ aims to
minimize the objective function~$f_i(x,y_i,y_{-i})$ subject to the
constraint
\begin{equation*}
  y_i\in K_i(x,y_{-i})\subseteq Y_i.
\end{equation*}
Thus, the followers solve the $x$-parametric GNEP given by the family
of problems
\begin{equation}
  \label{eq:SLMF-lower-problem}
  \min_{y_i}\bigl\{f_i(x,y_i,y_{-i}):y_i\in K_i(x,y_{-i})\bigr\},
  \quad i\in I.
\end{equation}
For each $x\in X$, we denote by $\GNEP(x)$ the set of equilibria
of~\eqref{eq:SLMF-lower-problem}. Equivalently, $y\in \GNEP(x)$ if and
only if, for every $i\in I$,
\begin{equation*}
  y_i\in K_i(x,y_{-i})
  \quad\text{and}\quad
  f_i(x,y_i,y_{-i})\leq f_i(x,z_i,y_{-i})
  \quad\forall z_i\in K_i(x,y_{-i}).
\end{equation*}
We also define
\begin{equation*}
  D\define\dom(\GNEP)=\{x\in X : \GNEP(x)\neq\emptyset\}.
\end{equation*}
The equilibrium problem~\eqref{eq:SLMF-lower-problem} is called the
followers' problem.

At the end of this work, we will pay particular attention to games that
are fully linear; in particular, they do not involve products of
decision variables across players. These games have the players'
interaction only in the constraints.
\begin{definition}[Fully linear games]\label{def:FullyLinear}
  We say that the equilibrium problem \eqref{eq:SLMF-lower-problem} is
  \emph{fully linear} if for every $i\in I$ we have that
  \begin{itemize}\setlength{\itemsep}{0.3cm}
  \item[(i)] $f_i(x,y_i,y_{-i}) = c_i^{\top}y_i$ for some vector $c_i\in\R^{q_i}$; and
  \item[(ii)] $K_i(x,y_{-i}) = \{y_i \in \R^{q_i} : A_ix + B_iy_i +
    C_{i}y_{-i}\leq b_i\} \subseteq Y_i$, for some matrices
    $A_i,B_i,C_i$ and a vector~$b_i$ of appropriate dimensions.
  \end{itemize}
\end{definition}

\subsection{Optimistic and Pessimistic Formulations}
\label{subsec:formulations}

In the optimistic formulation, the leader aims to minimize, in terms
of both $x$ and $y$ variables, an
objective function $F(x,y)$ while satisfying the coupling constraints
$G(x,y)\leq0$, where $G=(G_1,\ldots,G_m):X\times Y\to\R^m$ with $m
\geq 1$. The leader can select any equilibrium $y\in \GNEP(x)$ to
evaluate $F(x,\cdot)$ and $G(x,\cdot)$. The optimistic SLMF problem is
\begin{subequations}\label{eq:Optimistic-SLMF-problem}
  \begin{align}
    \min_{x,y}\quad &F(x,y)\label{eq:Optimistic-SLMF-objective}\\
    \text{s.t.}\quad &x\in X,\label{eq:Optimistic-SLMF-X}\\
                    &G(x,y)\leq0,\label{eq:Optimistic-SLMF-G}\\
                    &y\in \GNEP(x).\label{eq:Optimistic-SLMF-GNEP}
  \end{align}
\end{subequations}

In the pessimistic formulation, the equilibrium is selected
adversarially. The leader must therefore account for the worst value
of their objective function and must satisfy the coupling constraints
for every followers' equilibrium. Using direct epigraphical
reformulations as in \textcite{Wiesemann:2013:Pessimistic}, we can assume
that the objective function of the leader is a function
$\theta:X\to\R$ that does not depend on the followers
equilibrium. Consequently, throughout the rest of the paper we
consider the general pessimistic SLMF problem
\begin{subequations}
  \label{eq:Pessimistic-SLMF-problem}
  \begin{align}
    \min_x\quad
    &\theta(x)\label{eq:Pessimistic-SLMF-objective}\\
    \text{s.t.}\quad
    &x\in X\cap D,\label{eq:Pessimistic-SLMF-domain}\\
    &G(x,y) \leq 0,
      \quad \forall y\in \GNEP(x),\label{eq:Pessimistic-SLMF-robust}
  \end{align}
\end{subequations}

Note that in the pessimistic formulation, we need to make explicit the
constraint $x\in D =\dom(\GNEP)$. Otherwise, the leader might be
encouraged to choose a point $x\in X$ with $\GNEP(x) = \emptyset$ so
that the robust constraint~\eqref{eq:Pessimistic-SLMF-robust} is satisfied
by vacuity. This is not desirable in terms of modeling and so we
enforce the leader to be restricted only to those decisions for
which the followers' problem admits solutions. This constraint appears
implicitly in the optimistic formulation.

\subsection{Convexity and Linearity}
\label{subsec:convexity}

We say that the followers' problem is \emph{convex} if, for every
$i\in I$ and every $(x,y_{-i})\in X\times Y_{-i}$, the set $Y_i$ is convex, the function
\begin{equation*}
  f_i(x,\cdot,y_{-i}):Y_i\to\R
\end{equation*}
is convex, and the set $K_i(x,y_{-i})$ is convex.

We say that the optimistic SLMF
problem~\eqref{eq:Optimistic-SLMF-problem}  \emph{has convex data} if
$X$ is convex, $F$ and $G$ are jointly convex, and, for every $i\in
I$, the respective follower's problem is convex.

We say that the pessimistic SLMF
problem~\eqref{eq:Pessimistic-SLMF-problem}  \emph{has convex data} if
$X$ is convex, $\theta$ is convex, for every $i\in I$,
the respective follower's problem is convex, and the functions
$G_k(x,\cdot)$, $k\in [m]$, are concave, for every $x\in X\cap D$.
The concavity of the
robust coupling constraints is required due to the fact that there is an
implicit maximization problem given by $\sup_y\{G_k(x,y)\mid y\in
\GNEP(x)\}$. In this maximization problem, concavity of
$G_k(x,\cdot)$ is a desirable property; see e.g.,
\textcite{Wiesemann:2013:Pessimistic}.

Since an optimistic or pessimistic bilevel problem is the respective
SLMF game with just one follower, the definition of a bilevel problem
having convex data is included in the above definitions.

Now, if we are interested in linear formulations, we need to define
linear data. We say that the optimistic SLMF
problem~\eqref{eq:Optimistic-SLMF-problem} has linear data if
\begin{enumerate}
\item[(i)] The leader's objective function is of the form $F(x,y) =
  d^{\top}x + e^{\top}y$.
\item[(ii)] $X$ and $Y$ are polyhedra, i.e., sets
  defined by finitely many linear constraints; see,
  e.g., \textcite[Chapter 3.5]{Conforti-et-al:2014}.
\item[(iii)] The coupling constraint function is of the form $G(x,y) =
  Tx + Wy - h$ for some matrices $T,W$ and a vector $h$ of appropriate
  dimensions.
\item[(iv)] The lower-level GNEP is fully linear in the sense of
  Definition~\ref{def:FullyLinear}.
\end{enumerate}

Similarly, we say that the pessimistic SLMF
problem~\eqref{eq:Pessimistic-SLMF-problem} has linear data if the
leader's objective function is of the form $\theta(x) = d^{\top}x$
and if (ii)--(iv) hold as in the optimistic case.

\begin{remark}
  In the sequel, we will refer also to linear data for multilevel
  problems (trilevel, in particular).
  The notion extends in the same way as defined for SLMF problems, and it
  is already a known concept in multilevel programming; see, e.g.,
  \textcite{Jeroslow1985:Polynomial}.
\end{remark}

Of course, even with convex or linear data, both optimistic and pessimistic SLMF
problems may be nonconvex. Indeed, even with convex data, the sets
$\GNEP(x)$ are not necessarily convex, and the domain $D=\dom(\GNEP)$ might not
be convex either. Furthermore, even if $G$ is linear and $\GNEP(x)$ is
convex-valued, the equilibrium constraint \eqref{eq:Optimistic-SLMF-GNEP} and the robust
constraint \eqref{eq:Pessimistic-SLMF-robust} do not necessarily
define convex feasible sets.
The importance of convex data, at least in bilevel
optimization, is that it paves the way to several approaches to
either reformulate problems as single-level ones or to use different
algorithmic techniques to search for local or even global optima.

\section{Optimistic SLMF Problems as Bilevel Problems}
\label{sec:optimistic}

For $(x,\hat y)\in X\times Y$, let $S(x,\hat y)$ be the solution set
of the optimization problem
\begin{subequations}\label{eq:BP-lower-problem}
  \begin{align}
    \min_y\quad &\sum_{i\in I}f_i(x,y_i,\hat y_{-i})\label{eq:BP-lower-objective}\\
    \text{s.t.}\quad
                &y_i\in K_i(x,\hat y_{-i}),
                &&\forall i\in I,\label{eq:BP-lower-feasibility}\\
                &f_i(x,y_i,\hat y_{-i})
                  \leq f_i(x,\hat y_i,\hat y_{-i})
                &&\forall i\in I.\label{eq:BP-lower-improvement}
  \end{align}
\end{subequations}

\begin{lemma}\label{lem:fixed-point-GNEP}
  For every $x\in X$ and $y\in Y$,
  \begin{equation*}
    y\in \GNEP(x)
    \iff
    y\in S(x,y).
  \end{equation*}
  Equivalently, $\GNEP(x)$ is the set of fixed points of the
  set-valued map $S(x,\cdot):Y\tto Y$.
\end{lemma}

\begin{proof}
  Suppose first that $y\in \GNEP(x)$.
  Then, for every $i\in I$, one has \mbox{$y_i\in K_i(x,y_{-i})$} and
  \begin{equation*}
    f_i(x,y_i,y_{-i})\leq f_i(x,z_i,y_{-i})
    \quad
    \forall z_i\in K_i(x,y_{-i}).
  \end{equation*}
  Consequently, every feasible point $z$ of~\eqref{eq:BP-lower-problem}
  with $\hat y=y$ verifies $f_i(x,z_i,y_{-i})=f_i(x,y_i,y_{-i})$ for
  every $i\in I$. Therefore $f_i(x,\cdot,y_{-i})$ is constant in the
  feasible set.
  In particular, $y$ is feasible for~\eqref{eq:BP-lower-problem} and
  minimizes its objective function.
  Hence, $y\in S(x,y)$.

  Conversely, suppose that $y\in S(x,y)$ and assume, by contradiction,
  that \mbox{$y\notin \GNEP(x)$}. Since $y$ is feasible
  for~\eqref{eq:BP-lower-problem}, one has $y_i\in K_i(x,y_{-i})$, for
  every $i\in I$. Thus, for $y$ to fail being an equilibrium, there
  must exist $i\in I$ and $z_i\in K_i(x,y_{-i})$ such that
  \begin{equation*}
    f_i(x,z_i,y_{-i})<f_i(x,y_i,y_{-i}).
  \end{equation*}
  Noting that the vector $(z_i,y_{-i})$ is feasible
  for~\eqref{eq:BP-lower-problem} with $\hat y=y$, we can write
  \begin{equation*}
    f_i(x,z_i,y_{-i})
    +\sum_{j\in I\setminus\{i\}}f_j(x,y_j,y_{-j})
    <\sum_{j\in I}f_j(x,y_j,y_{-j}).
  \end{equation*}
  This contradicts $y\in S(x,y)$, finishing the proof.
\end{proof}

The fixed-point condition in Lemma~\ref{lem:fixed-point-GNEP} can be
lifted by introducing a copy $\hat y$ of the followers' strategy
vector, in a very similar vein to the approach of
\textcite{Zeng:2020:APractical}.

\begin{theorem}
  \label{thm:Reduction-SLMF-BP}
  The optimistic SLMF problem~\eqref{eq:Optimistic-SLMF-problem}
  admits the equivalent single-leader single-follower
  reformulation
  \begin{subequations}\label{eq:BP-upper-problem}
    \begin{align}
      \min_{x,\hat y,y}\quad &F(x,y)\label{eq:BP-upper-objective}\\
      \text{s.t.}\quad &x\in X,\label{eq:BP-upper-X}\\
                             &\hat y\in Y,\label{eq:BP-upper-Y}\\
                             &G(x,y)\leq0,\label{eq:BP-upper-G}\\
                             &\hat y=y,\label{eq:BP-upper-coupling}\\
                             &y\in S(x,\hat y),\label{eq:BP-upper-lower}
    \end{align}
  \end{subequations}
  where $S(x,\hat y)$ is the solution set
  of~\eqref{eq:BP-lower-problem}. More precisely,
  \begin{equation*}
    (x,y)\text{ solves~\eqref{eq:Optimistic-SLMF-problem}}
    \iff
    (x,y,y)\text{ solves~\eqref{eq:BP-upper-problem}}.
  \end{equation*}
\end{theorem}

\begin{proof}
  The two problems have the same upper-level objective function. It is
  therefore enough to compare their feasible points. By
  Lemma~\ref{lem:fixed-point-GNEP},
  \begin{align*}
    (x,y)\text{ is feasible for~\eqref{eq:Optimistic-SLMF-problem}}
    &\iff
      x\in X,\ G(x,y)\leq0,\ y\in S(x,y)\\
    &\iff
      (x,y,y)\text{ is feasible for~\eqref{eq:BP-upper-problem}}.
  \end{align*}
  This proves the equivalence.
\end{proof}

The next proposition shows that the reformulation of
Theorem~\ref{thm:Reduction-SLMF-BP} preserves convex/linear data and
that it can be done efficiently.

\begin{proposition}
  \label{prop:optimistic-reformulation-structure}
  Suppose that problem~\eqref{eq:Optimistic-SLMF-problem} has convex
  data.
  Then, the bilevel reformulation~\eqref{eq:BP-upper-problem}
  has convex data as well.

  If, in addition, problem~\eqref{eq:Optimistic-SLMF-problem} has
  linear data, then the bilevel
  reformulation~\eqref{eq:BP-upper-problem} has linear data. Moreover,
  under the standard explicit rational encoding of the linear data,
  the reformulation can be constructed in polynomial time and its
  encoding length is polynomially bounded by the encoding length of
  the original SLMF problem.
\end{proposition}
\begin{proof}
  The first statement follows directly from the construction in
  Theorem~\ref{thm:Reduction-SLMF-BP}. Indeed, for fixed
  $(x,\hat y)$, the objective function in~\eqref{eq:BP-lower-objective}
  is convex in $y$, and the constraints
  in~\eqref{eq:BP-lower-feasibility} and
  \eqref{eq:BP-lower-improvement} define convex sets in $y$.
  Moreover, the upper-level constraints inherited from
  problem~\eqref{eq:Optimistic-SLMF-problem} are convex, the
  restriction $\hat y\in Y$ is convex, and
  \eqref{eq:BP-upper-coupling} is linear.

  Now suppose that the original SLMF problem has linear data. Then
  the objective function in~\eqref{eq:BP-lower-objective} is fully
  linear (i.e., it is linear in all upper and lower variables),
  the constraints in~\eqref{eq:BP-lower-feasibility} are linear, and
  the optimality inequalities in~\eqref{eq:BP-lower-improvement}
  reduce to
  \begin{equation*}
    c_i^\top y_i \leq c_i^\top \hat y_i
    \quad \forall i\in I.
  \end{equation*}
  Hence, all upper- and lower-level objective functions and
  constraints of~\eqref{eq:BP-upper-problem} are linear.

  Finally, the reformulation introduces only one copy $\hat y$ of the
  followers' strategy vector and, for every follower, only copies of
  constraints already present in its optimization problem together
  with one additional linear optimality inequality. Consequently,
  the number of variables and constraints, as well as their encoding
  length, grows at most polynomially in the size of the original
  linear SLMF instance.
\end{proof}

The reformulation is obtained by lifting the fixed-point condition
$y\in S(x,y)$ into the two conditions $y\in S(x,\hat y)$ and $\hat
y=y$.
The coupling constraint~\eqref{eq:BP-upper-coupling} allows us to
impose the simple upper-level restriction $\hat y\in Y$. This avoids
adding fixed-point feasibility restrictions of the form
$\hat{y}_i\in K_i(x,\hat{y}_{-i})$, which may fail to be convex
constraints in the lifted upper-level variables $(x,\hat y)$.

\begin{remark}
  \label{rem:coupling-penalization}
  Theorem~\ref{thm:Reduction-SLMF-BP} states that an optimistic SLMF
  game can be transformed into an equivalent bilevel optimization
  problem.
  For solvable instances with rational linear data, applying
  \textcite{HenkeEtAl:2025:OnCoupling} yields a bilevel problem
  without coupling constraints that preserves linear data and has the
  same global optimal solutions after projection onto the original
  variables.
\end{remark}

\section{Pessimistic SLMF Problems as Optimistic Trilevel Programs}
\label{sec:pessimistic-to-three-level}

Our aim in this section is again to explore the potential of a
bilevel-type reformulation, but now for pessimistic
SLMF games. As shown by \textcite{Zeng:2020:APractical}, pessimistic
bilevel problems can be solved, under suitable assumptions, through a
tight optimistic bilevel relaxation and a subsequent recovery step.
Hence, it might be natural to expect the same result with multiple followers.
Thus, after Section~\ref{sec:optimistic} one would anticipate the
possibility of reducing pessimistic SLMF problems to bilevel
programming as well. However, in Section~\ref{subsec:tri_to_pess}
we will show that optimistic trilevel problems, whose rational linear
decision version is $\Sigma_2^p$-complete, see
\textcite{Jeroslow1985:Polynomial,Sugishita-Carvalho:2026},
can be equivalently reformulated as pessimistic SLMF problems.
Moreover, following the approach developed in
Section~\ref{sec:optimistic}, we will show in
Section~\ref{subsec:pess_to_tri} that the converse also holds, meaning
that any pessimistic SLMF (under a mild attainability hypothesis) can
be also equivalently reformulated as an optimistic trilevel problem.

\subsection{Trilevel Problems as Pessimistic SLMF Problems}
\label{subsec:tri_to_pess}

Consider now the trilevel problem
\begin{subequations}\label{eq:Tri-upper}
  \begin{align}
    \min_{x,y,z}\quad &F_1(x,y,z)\label{eq:Tri-upper-objective}\\
    \text{s.t.}\quad &x\in X,\label{eq:Tri-upper-X}\\
                      &H_1(x,y,z)\leq 0,\label{eq:Tri-upper-G}\\
                      &(y,z)\in S_1(x),\label{eq:Tri-upper-middle}
  \end{align}
\end{subequations}
where $S_1(x)$ is the solution set of the middle-level problem
\begin{subequations}\label{eq:Tri-middle}
  \begin{align}
    \min_{y,z}\quad &F_2(x,y,z)\label{eq:Tri-middle-objective}\\
    \text{s.t.}\quad &y\in Y(x),\label{eq:Tri-middle-Y}\\
                    &H_2(x,y,z)\leq0,\label{eq:Tri-middle-G}\\
                    &z\in S_2(x,y),\label{eq:Tri-middle-lower}
  \end{align}
\end{subequations}
and $S_2(x,y)$ is the solution set of the lower-level problem
\begin{subequations}\label{eq:Tri-lower}
  \begin{align}
    \min_z\quad &F_3(x,y,z)\label{eq:Tri-lower-objective}\\
    \text{s.t.}\quad &z\in Z(x,y).\label{eq:Tri-lower-Z}
  \end{align}
\end{subequations}

Our goal is to reduce this problem to a pessimistic SLMF problem. To
do so, we first need to encode the middle and lower levels as a
GNEP. To this end, fix $x\in X$ and $\hat{y}\in Y(x)$.
The feasible region of
problem~\eqref{eq:Tri-middle}--\eqref{eq:Tri-lower} is given by
\begin{equation*}
  W_0(x)\define \{(y,z):y\in Y(x),\ H_2(x,y,z)\leq0,\ z\in
  S_2(x,y)\}.
\end{equation*}
Let $W(x,\hat{y}) \define W_0(x)\times S_2(x,\hat{y})$.

\begin{lemma}\label{lemma:FeasibleRegionByGNEP}
  For every $x\in X$ and every $\hat{y}\in Y(x)$, the product
  $W(x,\hat{y})$ of the feasible region of the bilevel
  problem~\eqref{eq:Tri-middle}--\eqref{eq:Tri-lower}
  for $x$ together with the solution set of \eqref{eq:Tri-lower} for
  $(x,\hat{y})$ can be described as the equilibrium set of a
  two-player generalized Nash equilibrium problem.
\end{lemma}

\begin{proof}
  Consider the two-player generalized game given by
  \begin{equation}
    \label{eq:2-player-game}
    \begin{split}
      P_1(x,z)
      & \define
      \begin{cases}
        \displaystyle\min_y&0\\
        \text{s.t.}&y\in Y(x),\\
        &H_2(x,y,z)\leq0,
      \end{cases}
      \\
      P_2(x,y,\hat{y})
      & \define
      \begin{cases}
        \displaystyle\min_{(z,w)}&F_3(x,y,z) + F_3(x,\hat{y},w)\\
        \text{s.t.}&z\in Z(x,y), w\in Z(x,\hat{y}).
      \end{cases}
    \end{split}
  \end{equation}
  Then,
  \begin{align*}
    & (y,z,w)\text{ is an equilibrium for }(x,\hat{y})
    \\
    \iff & y\text{ solves }P_1(x,z)\text{ and }(z,w)\text{
           solves }P_2(x,y,\hat{y})
    \\
    \iff & y\in Y(x), H_2(x,y,z)\leq 0, z\in S_2(x,y)\text{ and }w\in
           S_2(x,\hat{y})
    \\
    \iff & (y,z,w)\in W(x,\hat{y}).
  \end{align*}
  Thus, the equilibrium set of the two-player game is exactly
  $W(x,\hat{y})$, which completes the proof.
\end{proof}

For each $x\in X$ and each $\hat{y}\in Y(x)$, let $W(x,\hat{y})$ be the
set of equilibria of the two-player game given by~\eqref{eq:2-player-game}.
We have that
\begin{equation}\label{eq:W-feasible-equivalence}
  (y,z,w)\in W(x,\hat{y})
  \iff
  (y,z)\text{ is feasible
    for~\eqref{eq:Tri-middle}--\eqref{eq:Tri-lower} and }w\in
  S_2(x,\hat{y}).
\end{equation}

Now, we propose a pessimistic SLMF problem in which the leader chooses a
copy~$(\hat y,\hat z)$ of the middle- and lower-level variables and
imposes its optimality in a robust sense against every $(y,z,w)\in
W(x,\hat{y})$.
With $D_W \define \dom(W) = \defset{(x,\hat{y})}{W(x,\hat{y}) \neq \emptyset}$,
this leads to the formulation
\begin{subequations}\label{eq:Tri-SLMF-upper}
  \begin{align}
    \min_{x,\hat y,\hat z}\quad
    &F_1(x,\hat y,\hat z)\label{eq:Tri-SLMF-upper-objective}\\
    \text{s.t.}\quad
    &x\in X,\label{eq:Tri-SLMF-upper-X}\\
    &\hat y\in Y(x),\label{eq:Tri-SLMF-upper-Y}\\
    &\hat z\in Z(x,\hat y),\label{eq:Tri-SLMF-upper-Z}\\
    &(x,\hat{y}) \in D_W,\\
    &H_1(x,\hat y,\hat z)\leq0,\label{eq:Tri-SLMF-upper-G1}\\
    &H_2(x,\hat y,\hat z)\leq0,\label{eq:Tri-SLMF-upper-G2}\\
    &F_2(x,\hat y,\hat z)-F_2(x,y,z)\leq0
    &\forall (y,z,w)\in W(x,\hat{y}),\label{Constraint:Robust-optimality-F2}\\
    &F_3(x,\hat y,\hat z)-F_3(x,\hat{y},w)\leq0
    &\forall (y,z,w)\in W(x,\hat{y}).\label{Constraint:Robust-optimality-F3}
  \end{align}
\end{subequations}

\begin{theorem}\label{thm:ThreeLevelToPessimistic}
  The trilevel problem~\eqref{eq:Tri-upper}--\eqref{eq:Tri-lower} and
  the pessimistic SLMF problem~\eqref{eq:Tri-SLMF-upper}, with
  $W(x,\hat{y})$ being defined as the equilibrium set of the
  two-player game given by~\eqref{eq:2-player-game}, are equivalent.
  More precisely, we have
  \begin{equation*}
    (x,y,z)\text{ solves~\eqref{eq:Tri-upper}--\eqref{eq:Tri-lower}}
    \iff
    (x,\hat y,\hat z) = (x,y,z)\text{ solves~\eqref{eq:Tri-SLMF-upper}}.
  \end{equation*}
\end{theorem}

\begin{proof}
  Fix $x\in X$. First, suppose that $(x,\hat y,\hat z)$ is
  feasible for~\eqref{eq:Tri-SLMF-upper}. Since
  $(x,\hat y)\in\dom(W)$, there exists
  $(y,z,w)\in W(x,\hat y)$.

  By~\eqref{eq:W-feasible-equivalence},
  we have $w\in S_2(x,\hat y)$. Moreover, $\hat z\in Z(x,\hat y)$ and
  the robust constraint~\eqref{Constraint:Robust-optimality-F3} imply
  \begin{equation*}
    F_3(x,\hat y,\hat z)
    \leq F_3(x,\hat y,w).
  \end{equation*}
  Since $w\in S_2(x,\hat y)$ and $\hat z\in Z(x,\hat y)$, the reverse
  inequality holds by optimality of $w$. Thus, we get
  $\hat z\in S_2(x,\hat y)$.
  Hence, together with
  constraints~\eqref{eq:Tri-SLMF-upper-Y} and \eqref{eq:Tri-SLMF-upper-G2}
  we get that $(\hat{y},\hat{z})$ is feasible for
  problems~\eqref{eq:Tri-middle}--\eqref{eq:Tri-lower}.
  By~\eqref{eq:W-feasible-equivalence},
  the robust constraint~\eqref{Constraint:Robust-optimality-F2} is
  equivalent to
  \begin{equation*}
    F_2(x,\hat y,\hat z)\leq F_2(x,y,z)
  \end{equation*}
  for every $(y,z)$ that is feasible for
  problems~\eqref{eq:Tri-middle}--\eqref{eq:Tri-lower}.
  Hence, we obtain
  \begin{equation*}
    (\hat y,\hat z)\in S_1(x).
  \end{equation*}

  Conversely, suppose that $(\hat y,\hat z)\in S_1(x)$. Then
  $(\hat y,\hat z)$ is feasible for the middle-level
  problem~\eqref{eq:Tri-middle}. In particular,
  \begin{equation*}
    \hat y\in Y(x),\qquad
    H_2(x,\hat y,\hat z)\leq 0,
    \qquad
    \hat z\in S_2(x,\hat y)\subseteq Z(x,\hat y).
  \end{equation*}
  Hence, $(\hat y,\hat z)\in W_0(x)$ and
  $S_2(x,\hat y)\neq\emptyset$. Therefore,
  \begin{equation*}
    W(x,\hat y)
    =W_0(x)\times S_2(x,\hat y)\neq\emptyset,
  \end{equation*}
  which yields $(x,\hat y)\in D_W$.

  Now, let $(y,z,w)\in W(x,\hat y)$. By~\eqref{eq:W-feasible-equivalence},
  the pair $(y,z)$ is feasible for the middle-level problem and
  $w\in S_2(x,\hat y)$. Since $(\hat y,\hat z)\in S_1(x)$, we have
  \begin{equation*}
    F_2(x,\hat y,\hat z)
    \leq
    F_2(x,y,z),
  \end{equation*}
  and thus the robust constraint~\eqref{Constraint:Robust-optimality-F2}
  is satisfied. Moreover, since both
  $\hat z,w\in S_2(x,\hat y)$, they attain the same lower-level
  objective value, and hence
  \begin{equation*}
    F_3(x,\hat y,\hat z)
    =
    F_3(x,\hat y,w).
  \end{equation*}
  Consequently, the robust
  constraint~\eqref{Constraint:Robust-optimality-F3} is satisfied as
  well.

  The remaining upper-level objective and constraints coincide with
  those of~\eqref{eq:Tri-upper} and, thus, the result follows.
\end{proof}

Let us now study the preservation of convex data of the latter theorem. We say that Problem \eqref{eq:Tri-upper}-\eqref{eq:Tri-middle}-\eqref{eq:Tri-lower} has
convex data if
\begin{itemize}
\item[(i)] The set $X$ is convex, and the functions
  $F_1(\cdot)$ and $H_1(\cdot)$ are convex in~$(x,y,z)$.
\item[(ii)] For every $x\in X$, the set $Y(x)$ is convex
  and the functions $F_2(x,\cdot,\cdot)$ and $H_2(x,\cdot,\cdot)$ are convex
  in~$(y,z)$.
\item[(iii)] For every $x \in X$ and every $y \in Y(x)$, the set $Z(x,y)$ is
  convex and the function $F_3(x,y,\cdot)$ is convex
  in~$z$.
\end{itemize}

  With this definition, if the trilevel problem~\eqref{eq:Tri-upper}-\eqref{eq:Tri-middle}-\eqref{eq:Tri-lower}
  has convex data, the robust
  constraints~\eqref{Constraint:Robust-optimality-F2}
  and~\eqref{Constraint:Robust-optimality-F3} are concave in the
  followers' variables $(y,z,w)$. However, the reformulation of
  Theorem~\ref{thm:ThreeLevelToPessimistic} might fail to preserve
  convex data due to the constraints
  \eqref{eq:Tri-SLMF-upper-Y}--\eqref{eq:Tri-SLMF-upper-Z}, which are
  not necessarily convex in $(x,\hat{y})$ and
  $(x,\hat{y},\hat{z})$ respectively.
  This can be addressed by modifying the first player
  of the GNEP~\eqref{eq:2-player-game}.

  Let $Y$ and $Z$ be convex ambient sets such that
  $Y(x)\subseteq Y$, for every $x\in X$, and
  $Z(x,y)\subseteq Z$, for every $x\in X$ and $y\in Y(x)$.
  We regard the functions $F_1,F_2,F_3,H_1,H_2$ as defined on
  $X\times Y\times Z$, with the convexity properties stated above.

  In the auxiliary game, for $\hat y\notin Y(x)$, define both
  followers' feasible sets to be empty. If $\hat y\in Y(x)$ but
  $y\notin Y(x)$, define the second follower's feasible set to be
  empty.

  For $\hat y\in Y(x)$, the first player solves the
  problem~$P_1(x,\hat{y},\hat{z},z)$ given by
  \begin{align*}
    \min_{(y,u,v)} \quad &0\\
    \text{s.t.} \quad &y\in Y(x),\\
                         &H_2(x,y,z)\leq0,\\
                         &u\in Y(x),\quad v\in Z(x,\hat{y}),\\
                         &u=\hat{y},\quad v=\hat{z}.
  \end{align*}
  Define
  \begin{equation*}
    \widetilde W(x,\hat{y},\hat{z})
    =
    \Defset{(y,z,u,v,w)}
    {
      \begin{array}{l}(y,u,v)\text{ solves }P_1(x,\hat{y},\hat{z},z),\\
        (z,w)\text{ solves }P_2(x,y,\hat{y})
      \end{array}
    }.
  \end{equation*}
  Then $\widetilde W$ encodes the
  constraints~\eqref{eq:Tri-SLMF-upper-Y}--\eqref{eq:Tri-SLMF-upper-Z}
  as feasibility constraints: if one of them fails, $P_1$ becomes
  infeasible. Thus,~\eqref{eq:Tri-SLMF-upper} can be equivalently
  rewritten as
  \begin{subequations}\label{eq:Tri-SLMF-upper2}
    \begin{align}
      \min_{x,\hat y,\hat z}\quad
      &F_1(x,\hat y,\hat z)\label{eq:Tri-SLMF-upper-objective2}\\
      \text{s.t.}\quad
      &(x,\hat y,\hat z)\in
        (X\times Y\times Z)\cap\dom(\widetilde W),
        \label{eq:Tri-SLMF-upper-XYZ}\\
      &H_1(x,\hat y,\hat z)\leq0,
        \label{eq:Tri-SLMF-upper-G12}\\
      &H_2(x,\hat y,\hat z)\leq0,
      &\forall (y,z,u,v,w)\in
        \widetilde W(x,\hat y,\hat z),
        \label{eq:Tri-SLMF-upper-G22}\\
      &F_2(x,\hat y,\hat z)-F_2(x,y,z)\leq0,
      &\forall (y,z,u,v,w)\in
        \widetilde W(x,\hat{y},\hat{z}),
        \label{Constraint:Robust-optimality-F22}\\
      &F_3(x,\hat y,\hat z)-F_3(x,\hat{y},w)\leq0,
      &\forall (y,z,u,v,w)\in
        \widetilde W(x,\hat{y},\hat{z}).
        \label{Constraint:Robust-optimality-F32}
    \end{align}
  \end{subequations}
  Since nonemptiness of $\widetilde W(x,\hat y,\hat z)$ is imposed,
  the robust constraint~\eqref{eq:Tri-SLMF-upper-G22} is equivalent
  to $H_2(x,\hat y,\hat z)\leq0$. Its left-hand side is constant,
  and hence concave, in the followers' variables.

  The leader's objective is convex, and the set
  \begin{equation*}
    \{(x,\hat y,\hat z)\in X\times Y\times Z:
    H_1(x,\hat y,\hat z)\leq0\}
  \end{equation*}
  is convex. The followers' objectives are convex in their own
  variables, and their feasible sets are convex under
  the assumptions above.
  Moreover, the left-hand sides of
  \eqref{Constraint:Robust-optimality-F22} and
  \eqref{Constraint:Robust-optimality-F32} are concave in the
  followers' variables. Thus, this pessimistic SLMF problem has
  convex data in the sense of Section~\ref{subsec:convexity}.

We close this section by completing the picture of reductions with
respect to linear data, as in
proposition~\ref{prop:optimistic-reformulation-structure} and forthcoming proposition
\ref{prop:pessimisticTotrivlevel-reformulation-structure}. We can
directly formulate the following proposition.

\begin{proposition}
  \label{prop:trivlevelTopessimistic-reformulation-structure}
  Suppose that Problem~\eqref{eq:Tri-upper}--\eqref{eq:Tri-lower}
  has linear data. Then, after removing terms independent of
  $(z,w)$ from the second follower's objective
  in~\eqref{eq:2-player-game}, the pessimistic SLMF
  formulation~\eqref{eq:2-player-game}--\eqref{eq:Tri-SLMF-upper}
  has linear data.

  Moreover, under the standard explicit rational encoding of
  the linear data, the pessimistic SLMF reformulation can be
  constructed in polynomial time and its encoding length is
  polynomially bounded by the encoding length of the original
  trilevel problem.
\end{proposition}
\begin{proof}
  Write
  \begin{equation*}
    F_3(x,y,z)=a^\top x+b^\top y+c^\top z.
  \end{equation*}
  The second follower's objective in~\eqref{eq:2-player-game}
  is then
  \begin{equation*}
    2a^\top x+b^\top y+b^\top\hat y+c^\top z+c^\top w.
  \end{equation*}
  For fixed $(x,y,\hat y)$, removing the terms independent of
  $(z,w)$ leaves the objective $c^\top z+c^\top w$ and does
  not change the set of minimizers. Together with the zero
  objective of the first follower and the linear constraints,
  this yields a fully linear game in the sense of
  Definition~\ref{def:FullyLinear}.

  The leader's objective and all explicit constraints
  in~\eqref{eq:Tri-SLMF-upper} are linear. Finally, the
  construction introduces only copies of variables and
  constraints from the original trilevel problem and performs
  a polynomial number of operations on their rational
  coefficients. Hence, its construction time and encoding
  length are polynomially bounded. The restriction to $D_W$
  is the standard equilibrium nonemptiness requirement and
  does not require an explicit enumeration of that set.
\end{proof}

\subsection{From Pessimistic SLMF Games to Trilevel Problems}
\label{subsec:pess_to_tri}

Consider the pessimistic SLMF
problem~\eqref{eq:Pessimistic-SLMF-problem}. For every $k\in [m]$,
introduce a copy $y^k\in\R^q$ of the followers' strategy vector and
consider
\begin{align*}
  \max_{r_k,y^k}\quad &r_k\\
  \text{s.t.}\quad &r_k-G_k(x,y^k)\leq0,\\
                      &y^k\in \GNEP(x).
\end{align*}
For fixed $x\in X\cap D$, this problem computes the maximum of
$G_k(x,\cdot)$ over $\GNEP(x)$ and represents it by the hypographical
variable $r_k$. If such maximum is attained, then the $k$th robust
constraint in~\eqref{eq:Pessimistic-SLMF-robust} can be represented by
the inequality $r_k\leq 0$. However, this might not be the case. To at
least rule out the problem when $G_k(x,\cdot)$ is unbounded on
$\GNEP(x)$ we can consider the truncated problem
\begin{equation}
  \label{eq:single-hypographical-problem-truncated}
  \begin{aligned}
    \max_{r_k,y^k}\quad &r_k\\
    \text{s.t.}\quad &r_k-G_k(x,y^k)\leq0,\\
    &r_k\leq 1,\\
    &y^k\in \GNEP(x).
  \end{aligned}
\end{equation}
Still, the $k$th robust constraint
in~\eqref{eq:Pessimistic-SLMF-robust} can be represented by the
inequality $r_k\leq 0$ but now under the relaxed condition that the
mapping $y\mapsto \min\{ G_k(x,y),1 \}$ attains its maximum on
$\GNEP(x)$.

Since the problems
in~\eqref{eq:single-hypographical-problem-truncated} are independent
for different indices $k\in [m]$, they can be collected into
\begin{subequations}\label{eq:Hypographical-middle}
  \begin{align}
    \max_{(r_k,y^k)_{k\in [m]}}\quad
    &\sum_{k=1}^m r_k\label{eq:Hypographical-objective}\\
    \text{s.t.}\quad
    &r_k-G_k(x,y^k)\leq 0,
    &\forall k\in[m],\label{eq:Hypographical-G}\\
    &r_k\leq 1
    &\forall k\in[m],\label{eq:truncation-G}\\
    &y^k\in \GNEP(x)
    &\forall k\in [m].\label{eq:multiple-selection-constraint}
  \end{align}
\end{subequations}
The equilibrium inclusions in~\eqref{eq:multiple-selection-constraint}
are independent copies of the same followers' problem. They can
therefore be reformulated simultaneously using the construction of
Section~\ref{sec:optimistic}.

\begin{theorem}\label{thm:Reduction-PessSLMF-3Level}
  Consider the pessimistic SLMF
  problem~\eqref{eq:Pessimistic-SLMF-problem}, and for each $k\in[m]$ define the map $\psi_k: (x,y)\mapsto \min\{G_k(x,y),1\}$. Assume that
  \begin{equation*}
    \label{hip:attainability}\tag{$\mathcal{H}$}
    \forall x\in X\cap D,\forall k\in[m]: \psi_k(x,\cdot)\text{
      attains its maximum over }\GNEP(x).
  \end{equation*}
  Then, Problem~\eqref{eq:Pessimistic-SLMF-problem} admits the
  equivalent trilevel reformulation
  \begin{subequations}\label{eq:SLMF-pessimistic-upper-tri}
    \begin{align}
      \min_{x,(y^k,\hat{y}^k,r_k)_{k\in[m]}}\quad
      &\theta(x)\label{eq:SLMF-pessimistic-upper-objective}\\
      \text{s.t.}\quad
      &x\in X,\label{eq:SLMF-pessimistic-upper-X}\\
      &r_k\leq 0 \quad\forall k\in[m],\label{eq:SLMF-pessimistic-upper-r}\\
      &(y^k,\hat{y}^k,r_k)_{k\in[m]}\in S_1(x),\label{eq:SLMF-pessimistic-upper-middle}
    \end{align}
  \end{subequations}
  where $S_1(x)$ is the solution set of the middle-level problem
  \begin{subequations}\label{eq:SLMF-pessimistic-middle-tri}
    \begin{align}
      \min_{(y^k,\hat{y}^k,r_k)_{k\in[m]}}\quad
      &-\sum_{k=1}^mr_k\label{eq:SLMF-pessimistic-middle-objective}\\
      \text{s.t.}\quad
      &r_k-G_k(x,y^k)\leq0
      &\forall k\in[m],\label{Constraint:ri-less-maxGi}\\
      &r_k\leq 1
      &\forall k\in[m],\label{Constraint:ri-truncation}\\
      & \hat{y}^k \in Y& \forall k \in [m],\\
      &\hat {y}^k=y^k
      &\forall k\in[m],\label{eq:SLMF-pessimistic-middle-coupling}\\
      &(y^k)_{k\in[m]}\in
        S_2\bigl(x,(\hat y^k)_{k\in[m]}\bigr),\label{eq:SLMF-pessimistic-middle-lower}
    \end{align}
  \end{subequations}
  and $S_2(x,(\hat y^k)_{k\in[m]})$ is the solution set of the lower-level problem
  \begin{subequations}\label{eq:SLMF-pessimistic-lower-tri}
    \begin{align}
      \min_{(y^k)_{k\in[m]}}\quad
      &\sum_{k\in[m]}\sum_{i\in I}
        f_i(x,y_i^k,\hat y_{-i}^k)\label{eq:SLMF-pessimistic-lower-objective}\\
      \text{s.t.}\quad
      &y_i^k\in K_i(x,\hat y_{-i}^k)
      &&\forall k\in[m],\ \forall i\in I,\label{eq:SLMF-pessimistic-lower-feasibility}\\
      &f_i(x,y_i^k,\hat y_{-i}^k)
        \leq f_i(x,\hat y_i^k,\hat y_{-i}^k)
      &&\forall k\in[m],\ \forall i\in I.\label{eq:SLMF-pessimistic-lower-improvement}
    \end{align}
  \end{subequations}
  More precisely, $x$ solves \eqref{eq:Pessimistic-SLMF-problem} if
  and only if there exist $(y^k,\hat{y}^k,r_k)_{k\in[m]}$ such that
  $(x,(y^k,\hat{y}^k,r_k)_{k\in[m]})$ solves
  \eqref{eq:SLMF-pessimistic-upper-tri}--\eqref{eq:SLMF-pessimistic-lower-tri}.
\end{theorem}
\begin{proof}
  The upper-level objective function is $\theta(x)$ in both
  formulations, so it is enough to prove the equivalence of their
  feasible leader decisions.

  Suppose first that $\big(x,(y^k,\hat{y}^k,r_k)_{k\in[m]}\big)$
  is feasible
  for~\eqref{eq:SLMF-pessimistic-upper-tri}--\eqref{eq:SLMF-pessimistic-lower-tri}.
  By~\eqref{eq:SLMF-pessimistic-middle-coupling},
  \eqref{eq:SLMF-pessimistic-middle-lower}, and the independence
  of the copied variables, Lemma~\ref{lem:fixed-point-GNEP} gives
  $y^k\in \GNEP(x)$ for each $k\in[m]$.
  By~\eqref{eq:SLMF-pessimistic-middle-coupling},
  \eqref{eq:SLMF-pessimistic-middle-lower}, and
  Lemma~\ref{lem:fixed-point-GNEP}, the feasible set of the
  middle-level problem is the Cartesian product, over $k\in[m]$, of
  the feasible sets of the independent problems
  \begin{equation*}
    \max_{r_k,y^k} \quad r_k \quad \text{s.t.} \quad
    r_k \leq G_k(x,y^k),\, r_k\leq 1, \, y^k\in\GNEP(x).
  \end{equation*}
  Since minimizing the middle-level objective
  $-\sum_{k=1}^m r_k$ is equivalent to maximizing
  $\sum_{k=1}^m r_k$, and the component problems are independent,
  middle-level optimality implies
  \begin{equation*}
    r_k= \max_{z\in \GNEP(x)} \psi_k(x,z)
    \quad
    \forall k\in[m].
  \end{equation*}
  Note that $\psi_k(x,y)\leq 0$ implies $G_k(x,y)\leq 0$.
  Thus, together with~\eqref{eq:SLMF-pessimistic-upper-r}, the above
  computation of $r_k$ yields $G_k(x,y)\leq0$, for every $y\in\GNEP(x)$
  and every $k\in[m]$.
  Therefore, $x$ is feasible
  for~\eqref{eq:Pessimistic-SLMF-problem}.

  Conversely, let $x$ be feasible
  for~\eqref{eq:Pessimistic-SLMF-problem}.
  By hypothesis~\eqref{hip:attainability}, for every $k\in[m]$
  there exists $y^k\in \GNEP(x)$ such that
  \begin{equation*}
    \max_{z\in \GNEP(x)}\psi_k(x,z) = G_k(x,y^k)\leq0.
  \end{equation*}
  For each $k\in [m]$, set
  $
  r_k \define \psi_k(x,y^k)$ and
  $\hat{y}^k\define y^k$.
  By Lemma~\ref{lem:fixed-point-GNEP}, each $y^k$ solves a
  corresponding copy of the lower-level
  problem~\eqref{eq:BP-lower-problem}, with $(x,\hat{y}) =
  (x,\hat{y}^k)$. Since these copies are independent, $(y^k)_{k\in
    [m]}$ solves~\eqref{eq:SLMF-pessimistic-lower-tri}, and so
  $(y^k,\hat{y}^k,r_k)_{k\in[m]}$ is feasible
  for~\eqref{eq:SLMF-pessimistic-middle-tri}. Moreover,
  \begin{equation*}
    -\sum_{k=1}^mr_k
    =
    -\sum_{k=1}^m\max_{z\in \GNEP(x)}\psi_k(x,z),
  \end{equation*}
  which is a lower bound for the objective value of every feasible
  point of the middle-level problem
  by~\eqref{Constraint:ri-less-maxGi} and~\eqref{Constraint:ri-truncation}. Therefore,
  $(y^k,\hat{y}^k,r_k)_{k\in[m]} \in S_1(x)$.
  Since $r_k\leq0$ holds for every $k\in [m]$, the resulting point is
  feasible
  for~\eqref{eq:SLMF-pessimistic-upper-tri}--\eqref{eq:SLMF-pessimistic-lower-tri}.
\end{proof}

\begin{remark}
  \label{rem:pessimistic-convexity}
  Under the assumptions of
  Theorem~\ref{thm:Reduction-PessSLMF-3Level}, suppose that the
  pessimistic SLMF problem has convex data.
  Assume additionally that $G_k(x,\cdot)$ is concave for every $k\in[m]$
  and every $x\in X$. Then, the trilevel reformulation in
  Theorem~\ref{thm:Reduction-PessSLMF-3Level} has convex data.
  Recall that convexity here is understood with the variables of preceding
  levels fixed, as in the definition given after
  Theorem~\ref{thm:ThreeLevelToPessimistic}.
\end{remark}

\begin{remark}
  Of course, the constant $1$ in the definition of $\psi_k$ can be
  replaced with any constant $c>0$.
\end{remark}

The attainability hypothesis \eqref{hip:attainability} is verified for
example, whenever $\GNEP(\cdot)$ has compact values and the functions
$G_k(x,\cdot)$ are continuous, for every $x\in X\cap D$. But also it is
trivially verified whenever $\sup\{G_k(x,y)\,:\, y\in \GNEP(x)\} >
1$. In general, the problem is then reduced to the case when
\begin{equation*}
  \Defset{ r\in \R}{ \exists y\in \GNEP(x),\, r=G_k(x,y) } \subset ]-\infty,1],
\end{equation*}
on which we need to verify that $G_k(x,\cdot)$ attains its maximum
directly. The truncation technique reduces the verification of the
attainability hypothesis~\eqref{hip:attainability} to cases in which
$G_k(x,\cdot)$ is bounded above on $\GNEP(x)$. Under linear data,
every such finite supremum is attained, as shown next.

\begin{lemma}\label{lemma:Attainability-with-linearData}
  Suppose that Problem~\eqref{eq:Pessimistic-SLMF-problem} has linear
  data. Then, the hypothesis~\eqref{hip:attainability} is verified.
\end{lemma}
\begin{proof}
  Fix $x\in X\cap D$ and $k\in [m]$. Since $x\in D$, $\GNEP(x)$
  is nonempty.  Let $\alpha = \sup\{ G_k(x,y)\,:\, y\in\GNEP(x)\}$
  which is either finite or $+\infty$. If $\alpha>1$, there must
  exist $\bar{y}\in \GNEP(x)$ such that $G_k(x,\bar{y})>1$ and so
  $\psi_k(x,\cdot)$ attains its maximum at $\bar{y}$. Thus, we only
  need to study the case in which $\alpha\leq 1$ holds.

  Since the equilibrium problem~\eqref{eq:SLMF-lower-problem} is fully
  linear, $\GNEP(x)$ can be written as a finite union of
  nonempty polyhedra $P_1,\ldots, P_s$.
  Indeed, for fixed $x$, the linear-program optimality conditions
  of all followers form a system of linear constraints and
  complementarity conditions. Resolving each complementarity
  condition into its two zero alternatives gives a finite union
  of polyhedra in the strategy and multiplier variables.
  Projecting out the multipliers yields the claimed representation
  of $\GNEP(x)$.
  Let $t\in [s]$. Since $G_k(x,\cdot)$ is
  linear and
  \begin{equation*}
    \sup \Defset{ G_k(x,y)}{ y \in P_t }\leq \alpha <+\infty,
  \end{equation*}
  we get that there must exist $y_t\in P_t$ such that
  $G_k(x,\cdot)$ attains its maximum on $P_t$ at $y_t$. Then, by
  choosing
  \begin{equation*}
    \bar{y} \in \argmax \Defset{ G_k(x,y)}{ y\in \{y_t\,:\, t=1,\dots,s\} },
  \end{equation*}
  we get that $G_k(x,\bar{y}) = \alpha$. Thus, $\psi_k(x,\cdot)$
  attains its maximum on $\GNEP(x)$ at $\bar{y}$, which completes the
  proof.
\end{proof}

Using the lemma above and following the same vein of arguments as in
Proposition~\ref{prop:optimistic-reformulation-structure}, we can
directly formulate the following proposition.

\begin{proposition}\label{prop:pessimisticTotrivlevel-reformulation-structure}
  Suppose that Problem~\eqref{eq:Pessimistic-SLMF-problem} has linear
  data. Then, the hypothesis of
  Theorem~\ref{thm:Reduction-PessSLMF-3Level} holds, and the trilevel
  formulation~\eqref{eq:SLMF-pessimistic-upper-tri}--\eqref{eq:SLMF-pessimistic-lower-tri}
  has linear data.

  Moreover, under the standard explicit rational encoding of the linear data,
  the trilevel reformulation can be constructed in polynomial time and its
  encoding length is polynomially bounded by the encoding length of
  the original pessimistic SLMF problem.
\end{proposition}

\section{Complexity Analysis under Rational Linear Data}
\label{sec:complexity-under-linear-data}

If we restrict our attention to rational linear data, then it is possible to
perform complexity analysis, following the classic results from
\textcite{Jeroslow1985:Polynomial} and more recently from
\textcite{BUchheim2023:BilevelinNP} and \textcite{Sugishita-Carvalho:2026}.

Recall that for an optimization problem of the form $\min\{ f(x)\,:\,
x\in X \}$, the associated decision problem is, for a given value
$t\in \Q$, to decide whether there exists $x\in X$ such that $f(x)\leq
t$.

\begin{theorem}\label{thm:completexity-rationallineardata}
  The decision problem associated to  optimistic SLMF problem with
  rational linear data is NP-complete, while the decision problem associated to
  pessimistic SLMF problem with rational linear data is $\Sigma_2^p$-complete.
\end{theorem}
\begin{proof}
  Clearly, on the one hand, linear bilevel optimization is a subclass
  of optimistic SLMF problems with linear data.
  This yields that the decision problem of the latter class is
  NP-hard; see \textcite{Jeroslow1985:Polynomial}.
  On the other hand, Theorem~\ref{thm:Reduction-SLMF-BP} and
  Proposition~\ref{prop:optimistic-reformulation-structure}
  provide a polynomial-time reduction from optimistic SLMF
  problems with rational linear data to rational linear bilevel
  optimization problems, preserving feasible objective values.
  Thus, the decision version of the optimistic SLMF problem with
  linear data is NP-complete as well, since the decision version of
  rational linear bilevel optimization is NP-complete; see
  \textcite{Jeroslow1985:Polynomial} and
  \textcite{BUchheim2023:BilevelinNP}.

  If one considers trilevel optimization problems with rational linear
  data, their decision version is $\Sigma_2^p$-complete:
  $\Sigma_2^p$-hardness follows from
  \textcite{Jeroslow1985:Polynomial}, while $\Sigma_2^p$-completeness
  follows from \textcite{Sugishita-Carvalho:2026}.

  Then
  Proposition~\ref{prop:trivlevelTopessimistic-reformulation-structure}
  and
  Proposition~\ref{prop:pessimisticTotrivlevel-reformulation-structure},
  yield directly that the decision version of the pessimistic SLMF
  problem with rational linear data is $\Sigma_2^p$-complete.
\end{proof}

In this discussion, we did not analyze pessimistic SLMF problems
without coupling constraints.
In general, such a problem has the form of
\begin{equation}
  \label{eq:Pessimistic-SLMF-problem-NoCoupling}
  \min_{x\in X\cap D} \sup_{y\in \GNEP(x)} F(x,y).
\end{equation}
The supremum is understood in the extended-real sense and may
equal $+\infty$. For a threshold $\tau\in\Q$, the associated
decision problem asks whether there exists $x\in X\cap D$ such
that $F(x,y)\leq\tau$ for every $y\in\GNEP(x)$.
Using the epigraphical reformulation of
\textcite{Wiesemann:2013:Pessimistic}, this is equivalent to asking
whether the following problem has a feasible point with
$t\leq\tau$:
\begin{subequations}
  \label{eq:Pessimistic-SLMF-epi}
  \begin{align}
    \min_{(x,t)}\quad
    &t\label{eq:Pessimistic-SLMF-epi-objective}\\
    \text{s.t.}\quad
    &x\in X\cap D, t\in\R,\label{eq:epi-domain}\\
    &F(x,y) - t\leq0
      \quad \forall y\in \GNEP(x).\label{eq:Epi-robust}
  \end{align}
\end{subequations}

Under rational linear data, the epigraphical reformulation
\eqref{eq:Pessimistic-SLMF-epi} is a polynomial-time reduction to
the decision problem in Theorem~\ref{thm:completexity-rationallineardata}.
Consequently, the decision version of
\eqref{eq:Pessimistic-SLMF-problem-NoCoupling} belongs to
$\Sigma_2^p$ as a particular subclass of pessimistic SLMF problems.

However, this does not settle its hardness. It is NP-hard,
since it contains pessimistic linear bilevel optimization
without coupling constraints as a special case.
To establish $\Sigma_2^p$-hardness of
Problem~\eqref{eq:Pessimistic-SLMF-problem-NoCoupling}, a first idea is
to reduce problem~\eqref{eq:Pessimistic-SLMF-problem} to a pessimistic SLMF
problem without coupling constraints using exact penalization, in the
vein of
\textcite{LefebvreSchmidt2025:Exact,HenkeEtAl:2025:OnCoupling,HenkeEtAl:2026:CouplingPessimistic}. The
idea would be to find, under linear data, a factor $\rho>0$ large
enough such that
\begin{equation*}
  \min_{x\in D\cap X} \theta(x) +\rho \sup_{y\in \GNEP(x)} \sum_{k=1}^m [G_k(x,y)]_+
\end{equation*}
has the same solutions as~\eqref{eq:Pessimistic-SLMF-problem}, where
$[a]_+\define \max\{0,a\}$, and then modify the equilibrium problem to
replace $[G_k(x,y)]_+$ by slack variables, as in
\textcite{HenkeEtAl:2025:OnCoupling,HenkeEtAl:2026:CouplingPessimistic}. The
second step is easy to do. However, the penalization step is not
always possible as the following example shows.

\begin{example}\label{ex:PenalizationDoesNotWork}
  Consider the following GNEP of four players parameterized by
  \mbox{$x\in X = [0,1]^2$}:
  \begin{subequations}
    \begin{align}
      & \min_{a\in \R} \quad a \quad \text{s.t.} \quad a \geq b+x_1-1,
        \,a \geq b-x_1,
        \label{eq:aux-p1}\tag{P$_1$}\\
      & \min_{b\in \R} \quad b \quad \text{s.t.} \quad b\geq a,
        \label{eq:aux-p2}\tag{P$_2$}\\
      & \min_{c\in \R} \quad c \quad \text{s.t.} \quad 0\leq c \leq 1,
        \,c \geq d-x_2,
        \label{eq:aux-p3}\tag{P$_3$}\\
      & \min_{d\in \R} \quad d \quad \text{s.t.} \quad 0\leq d\leq 1, \, d\geq c.
        \label{eq:aux-p4}\tag{P$_4$}
    \end{align}
  \end{subequations}

  Note first that at equilibrium between~\eqref{eq:aux-p1}
  and~\eqref{eq:aux-p2}, we have that
  \begin{equation*}
    a = b+\max\{x_1-1,-x_1\}\quad\text{ and }\quad b=a.
  \end{equation*}
  This entails that $\max\{x_1-1,-x_1\} = 0$ and so
  $x_1\in\{0,1\}$. The equilibrium between~\eqref{eq:aux-p3}
  and~\eqref{eq:aux-p4} holds when
  \begin{equation*}
    c = \max\{0,d-x_2\}\quad \text{ and }\quad d = c.
  \end{equation*}
  Thus, $c = \max\{0,c-x_2\}$. If $x_2 = 0$, then we can choose
  $c=d\in [0,1]$ and satisfy the latter equation. On the contrary, if
  $x_2 > 0$, then the only value that verifies the equations is
  $c=d=0$. Thus, we deduce that
  \(D = \dom (\GNEP) = \{0,1\}\times [0,1]\) and that for $x\in D$,
  \begin{equation*}
    \GNEP(x) = \begin{cases}
      \{(a,a,0,0)\,:\, a\in \R\}, &\text{if } x_2>0,\\
      \{(a,a,c,c)\,:\, a\in \R, c\in [0,1]\}, &\text{if } x_2=0.\\
    \end{cases}
  \end{equation*}
  Now, writing $y=(a,b,c,d)$, consider the robust constraint
  \begin{equation*}
    G(x,y)  = x_2 - x_1 + c \leq 0 \quad \forall y\in \GNEP(x),
  \end{equation*}
  and the objective function $\theta(x) = x_1$. Then, the pessimistic
  SLMF problem~\eqref{eq:Pessimistic-SLMF-problem} is given by
  \begin{equation}
    \label{eq:PessimisticProblem-nonpenalizable}
    \begin{split}
      \min_{x} \quad & x_1 \\
      \text{s.t.} \quad & x \in \{0,1\}\times [0,1],\\
      & x_2-x_1 + c\leq 0 \quad \forall c\in C(x),
    \end{split}
  \end{equation}
  where $C(x)$ is the projection to the third coordinate of $\GNEP(x)$,
  that is, $C(x) = \{0\}$ if $x_2>0$ and $C(x) = [0,1]$ if $x_2 = 0$.
  Note that the optimal value for this problem is $\theta^*=1$.
  Indeed, if $x_1=0$, the robust constraint is violated:
  for $x_2>0$, choose $c=0$, and for $x_2=0$, choose $c=1$.
  Now, take any $\rho>0$ and consider the penalized problem
  \begin{equation}
    \label{eq:aux-p-rho}
    \tag{P$_\rho$}
    \min_{x\in \{0,1\}\times [0,1]} v(x) \define x_1
    + \rho\max_{c\in C(x)} [x_2-x_1 + c]_+,
  \end{equation}
  where $[x_2-x_1 + c]_+ = \max\{ 0,x_2-x_1 + c\}$. On the one hand,
  in the branch $\{1\}\times [ 0,1]$, the robust constraint of
  \eqref{eq:PessimisticProblem-nonpenalizable} is verified, and so
  $v(x) = 1$. On the other hand, in the branch $\{0\}\times [0,1]$, we
  get that
  \begin{equation*}
    v(x) =
    \begin{cases}
      \rho x_2, &\text{if }x_2>0,\\
      \rho, &\text{if }x_2 = 0.
    \end{cases}
  \end{equation*}
  Thus, by taking $x_2\to 0$, we deduce that the infimum value of $v$
  is 0, but it is not attained. Thus, regardless of the value of $\rho$,
  Problem~\eqref{eq:aux-p-rho} does not admit solutions and its infimum
  value is different from the optimal value of
  \eqref{eq:PessimisticProblem-nonpenalizable}.\hfill$\Diamond$
\end{example}

Even though it is not possible to
reduce~\eqref{eq:Pessimistic-SLMF-problem}
to~\eqref{eq:Pessimistic-SLMF-problem-NoCoupling} via direct
penalization of the robust constraints, we still can show that the
decision version of pessimistic SLMF problems with rational linear
data and no coupling constraints is $\Sigma_2^p$-hard. We reduce
the DeNegre variant of the bilevel knapsack problem introduced in
\textcite{DeNegre:2011}, which is known to be $\Sigma_2^p$-complete;
see \textcite{Caprara-et-al:2014:Knapsack}. For positive integers
$A,B, a_1,\ldots,a_n,b_1,\ldots,b_n$, the DeNegre formulation is given
by
\begin{subequations}\label{eq:DNeg-upper}
  \begin{align}
    \min_{x,y}\quad &\sum_{k=1}^n b_ky_k\\
    \text{s.t.}\quad &\sum_{k=1}^n a_kx_k \leq A, \label{eq:KnapsackLeader}\\
                    &x \in \{0,1\}^n,\label{eq:IntegralityLeader}\\
                    &y \in S(x),
  \end{align}
\end{subequations}
where $S(x)$ is the solution set of
\begin{subequations}\label{eq:DNeg-lower}
  \begin{align}
    \max_{y}\quad &\sum_{k=1}^n b_ky_k\\
    \text{s.t.}\quad &\sum_{k=1}^n b_ky_k\leq B, \label{eq:KnapsackFollower}\\
                  &y_k \leq 1- x_k \quad\forall k\in[n],\label{eq:InterdictionFollower}\\
                  &y \in \{0,1\}^n.\label{eq:IntegralityFollower}
  \end{align}
\end{subequations}
In this problem, each player has their private knapsacks: the leader
picks some items to pack and then the follower chooses from the
remaining items. The leader aims to minimize the gain of the follower.

\begin{theorem}\label{thm:complexity-nocoupling}
  The decision problem associated to pessimistic SLMF problem with
  rational linear data and no coupling constraints is
  $\Sigma_2^p$-complete.
\end{theorem}
\begin{proof}
  We only need to reduce the DeNegre bilevel knapsack problem to
  pessimistic SLMF problems with rational linear data.
  Let $A,B,a_1,\ldots,a_n,b_1,\ldots,b_n\in\N$ and consider the
  DeNegre bilevel knapsack problem
  \eqref{eq:DNeg-upper}--\eqref{eq:DNeg-lower}.

  Set $X = \{ x\in \R^n\,:\, x\in[0,1]^n, \sum_{k=1}^n a_kx_k \leq A
  \}$. Then, \eqref{eq:DNeg-upper}--\eqref{eq:DNeg-lower} can be
  written as
  \begin{equation}\label{eq:First-reformualtion}
    \min_{x\in X\cap \{0,1\}^n} \max_{y\in Y(x)} \sum_{k=1}^n b_ky_k,
  \end{equation}
  where $Y(x)$ is the feasible set of~\eqref{eq:DNeg-lower}. Now,
  consider the two-player GNEP given by the problems
  \begin{align*}
    \min_y \quad & 0\\
    \text{s.t.} \quad & \sum_{k=1}^n b_ky_k\leq B, y\leq 1-x,\\
                 &y\in [0,1]^n,\\
                 &u\leq 0, v\leq 0
  \end{align*}
  and
  \begin{align*}
    \min_{(u,v)} \quad & -\sum_{k=1}^n u_k - \sum_{k=1}^n v_k\\
    \text{s.t.} \quad & u \leq 1-y, u \leq y,\\
                       &v \leq 1-x, v\leq x.
  \end{align*}
  At equilibrium, the unique best response of the second player is
  $u = \min\{1-y,y\}$ and $v = \min\{1-x,x\}$, where the minima
  are taken componentwise. Since $x,y\in[0,1]^n$, these vectors
  are nonnegative. Thus, the first player's constraints
  $u\leq0$ and $v\leq0$ imply $u=v=0$ and force $x$ and $y$
  to be binary.

  Conversely, for every $x\in X\cap\{0,1\}^n$ and every
  $y\in Y(x)$, the tuple $(y,0,0)$ is an equilibrium:
  the first player is feasible and has constant objective,
  and $(0,0)$ is the second player's unique best response.
  Consequently, for every $x\in X$, we have:
  \begin{equation*}
    \GNEP(x) = \left\{\begin{array}{cl}
                        \left\{ (y,0,0)\,:\,  \begin{array}{l}
                                                \sum_{k=1}^n b_ky_k\leq B,\\
                                                y\leq 1-x,\, y\in \{0,1\}^n
                                              \end{array} \right\},\quad &\text{if }x\in \{0,1\}^n,\\
                        \emptyset,&\text{otherwise.}
                      \end{array}\right.
                  \end{equation*}
  Since $0\in Y(x)$ for every $x\in X\cap\{0,1\}^n$, we obtain
  $D=\dom(\GNEP)=X\cap\{0,1\}^n$.
  We can therefore rewrite~\eqref{eq:First-reformualtion} directly as
  \begin{equation*}
    \min_{x\in X\cap D} \max_{(y,u,v)\in \GNEP(x)} \sum_{k=1}^n b_ky_k,
  \end{equation*}
  which has exactly the form of a pessimistic SLMF problem with
  rational linear data and no coupling constraints. Since the
  construction is readily polynomial in the encoding of
  $A,B,a_1,\ldots,a_n,b_1,\ldots,b_n$, the proof is completed.
\end{proof}

\section{Conclusions}
\label{sec:conclusions}

In this paper, we studied the relationship between
SLMF problems and bilevel optimization. We
showed that every optimistic SLMF problem admits an equivalent optimistic bilevel
reformulation. For pessimistic SLMF problems, we established
reformulations to and from optimistic trilevel optimization, with an
attainability assumption required in the forward direction. Under
rational linear data, this assumption always holds and the
reformulations preserve linearity and have polynomial encoding length.
The associated decision problems are NP-complete in the optimistic
case and $\Sigma_2^p$-complete in the pessimistic case, even without
coupling constraints and with only two followers.

These results clarify when methods for bilevel optimization can be
applied to SLMF problems and when an additional optimization level is
needed. A natural next step would be to investigate the computational
performance of the proposed reformulations. It is also of interest to
identify subclasses of pessimistic SLMF problems that admit
reformulations as linear bilevel problems despite the complexity
obstruction for the general class.

\section*{Acknowledgements}

The second author was partially supported by CMM BASAL Grant FB210005
and FONDECYT Grant 1251159 (ANID-Chile).

\section*{Data Availability Statement}

No Data was used during this research.
\section*{Use of AI}

OpenAI's GPT was used to improve writing and presentation of the
manuscript, and it helped for preliminary checking of correctness of
the proofs. All the mathematical content (statements and proofs) are
due to the authors, with the exception of Example 5.2, which was obtained
using GPT.

\printbibliography

\end{document}